\documentclass[12pt]{amsart}
\usepackage[all]{xy}
\usepackage{graphicx}
\usepackage{tikz}

\usepackage{amsmath}
\usepackage{amssymb}
\usepackage{amsfonts}
\usepackage{mathrsfs}
\usepackage{mathtools}

\newcommand{\Cdb}{\ensuremath{\mathbb{C}}}
\newcommand{\Ddb}{\ensuremath{\mathbb{D}}}

\newcommand{\Tdb}{\ensuremath{\mathbb{T}}}

\newcommand{\Zdb}{\ensuremath{\mathbb{Z}}}

\renewcommand{\P}{\mbox{${\mathcal P}$}}

\newcommand{\norm}[1]{\Vert#1\Vert}
\newcommand{\bignorm}[1]{\bigl\Vert#1\bigr\Vert}
\newcommand{\Bignorm}[1]{\Bigl\Vert#1\Bigr\Vert}

\newtheorem{theorem}{Theorem}[section]
\newtheorem{lemma}[theorem]{Lemma}
\newtheorem{corollary}[theorem]{Corollary}
\newtheorem{proposition}[theorem]{Proposition}
\newtheorem{definition}[theorem]{Definition}
\theoremstyle{remark}

\theoremstyle{definition}

\numberwithin{equation}{section}

\begin{document}

\title[Families of polygonal type operators]
{Commuting families of polygonal type operators on Hilbert space}

\author[C. Le Merdy]{Christian Le Merdy}
\email{clemerdy@univ-fcomte.fr}
\address{Laboratoire de Math\'ematiques de Besan\c con, 
Universit\'e de Franche-Comt\'e, 16 route de Gray
25030 Besan\c{c}on Cedex, FRANCE}

\author[M. N. Reshmi ]{M.N. Reshmi}
\email{rmazhava@math.univ-toulouse.fr}
\address{Institute de Math\'ematiques Toulouse,
Universit\'e  Paul Sabatier, 118, route de Narbonne
F-31062 Toulouse Cedex 9,   FRANCE}

\date{\today}

\begin{abstract} 
Let $T\colon H\to H$ be a bounded operator on Hilbert space $H$. We say that 
$T$ has a polygonal type if  there exists an open convex polygon 
$\Delta\subset {\mathbb D}$, with $\overline{\Delta}\cap{\mathbb T}\neq\emptyset$,
such that the spectrum $\sigma(T)$ is included in $\overline{\Delta}$ and the resolvent
$R(z,T)$ satisfies an estimate
$\Vert R(z,T)\Vert \lesssim \max\{\vert z-\xi\vert^{-1}\, :\, \xi\in 
\overline{\Delta}\cap{\mathbb T}\}$ for 
$z\in\overline{\mathbb D}^c$. The class of polygonal type operators 
(which goes back to De Laubenfels and Franks-McIntosh) contains the class of Ritt operators.
Let $T_1,\ldots,T_d$ be commuting operators on $H$, with $d\geq 3$. We prove functional calculus properties
of the $d$-tuple $(T_1,\ldots,T_d)$ under various assumptions involving poygonal type. The main ones 
are the following.
(1) If the operator $T_k$ is a contraction for all $k=1,\ldots,d$ and if $T_1,\ldots,T_{d-2}$ have a polygonal type, 
then $(T_1,\ldots,T_d)$ satisfies a generalized von Neumann inequality 
$\norm{\phi(T_1,\ldots,T_d)}\leq C\norm{\phi}_{\infty,{\mathbb D}^d}$ for polynomials $\phi$ in $d$ variables;
(2) If $T_k$ is polynomially bounded with a polygonal type for all $k=1,\ldots,d$, then 
there exists an invertible operator $S\colon H\to H$ such that $\norm{S^{-1}T_kS}\leq 1$
for all $k=1,\ldots,d$.
\end{abstract}

\maketitle



\vskip 1cm

\section{Introduction}\label{1Intro}   
The starting point of this paper is the following famous open problem in operator theory. Let $H$ be a Hilbert space, let 
$d\geq 3$ be an integer and let $T_1,T_2,\ldots,T_d\colon H\to H$ be commuting linear contractions. Is there
a constant $C\geq 1$ such that
\begin{equation}\label{1VN}
\norm{\phi(T_1,\ldots,T_d)}\leq C\sup\bigl\{\vert\phi(z_1,\ldots,z_d)\vert\,:\,
(z_1,\ldots,z_d)\in\Ddb^d\bigr\} 
\end{equation}
for all complex polynomials $\phi$ in $d$ variables?
Here $\Ddb=\{z\in\Cdb\, :\, \vert z\vert<1\}$ denotes the open unit disc
of the complex plane.   Von Neumann's inequality and Ando's 
inequality    assert that (\ref{1VN}) holds true for $d=1$ and 
$d=2$, respectively, with $C=1$. However not much is known 
for $d\geq 3$. 
We refer to \cite{Varo}, \cite[Chapter 1]{P0}, \cite{GS} and \cite{GR} for information and references on this issue.
 An estimate (\ref{1VN}) is usually called a generalized von Neumann inequality.

The main objective of this paper is to prove such an estimate in the case when the
contractions $T_k$ belong to the class of polygonal type operators. 
Let $\Tdb=\{z\in\Cdb\, :\, \vert z\vert=1\}$ denote the unit circle. 
We say that a bounded operator $T\colon H\to H$
has a polygonal type if there exists an open convex polygon 
$\Delta\subset\Ddb$, with $\overline{\Delta}\cap\Tdb\neq\emptyset$,
such that the spectrum $\sigma(T)$ is included in $\overline{\Delta}$ and the resolvent
$R(z,T)$ satisfies an estimate
$$
\norm{R(z,T)}\lesssim \max\bigl\{\vert z-\xi\vert^{-1}\, :\, \xi\in 
\overline{\Delta}\cap\Tdb\bigr\},\qquad
z\in\overline{\Ddb}^c.
$$
Note that this implies that  $\sigma(T)\cap\Tdb$ is finite.

\begin{figure}[!h]
\includegraphics[scale=0.25]{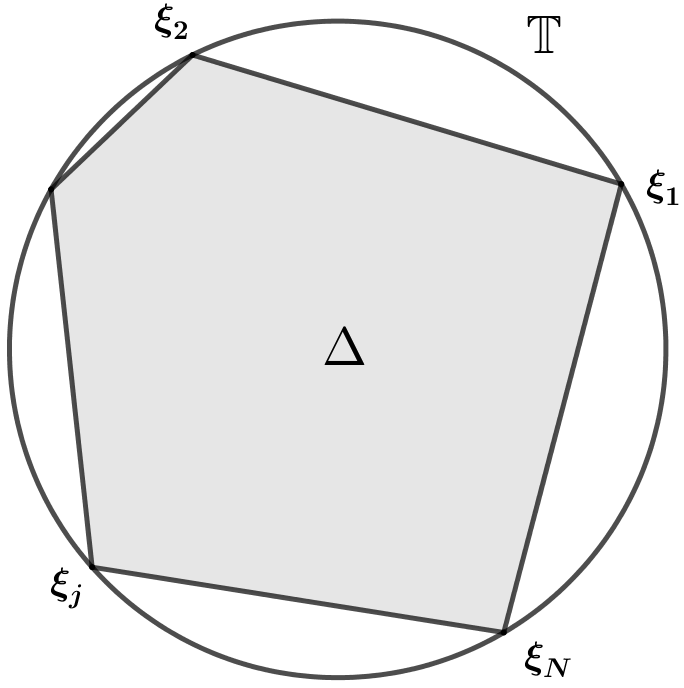}
\caption{Polygon containing the spectrum of $T$}
\label{fig:polygon}
\end{figure}

Polygonal type operators were introduced (independently) by De Laubenfels in \cite{dL} and 
by Franks and McIntosh in \cite{FMI}, in order to give a partial solution to the so-called Halmos problem.
Recall that in \cite{Hal}, Halmos had asked whether any polynomially bounded
operator  $T\colon H\to H$ is necessarily similar to a contraction. 
Pisier solved this question negatively  (see \cite{P1} and  \cite[Chapter 9]{P0}). 
Pisier's result was followed by \cite{dL,FMI}, in which 
it is proved that if $T\colon H\to H$
has a polygonal type, then $T$ is polynomially bounded (if and) only if $T$ is similar to a contraction.
It is also established in  \cite{dL} that if $T$ is a polynomially bounded operator 
with a polygonal type, then there exist an open convex polygon 
$\Delta\subset\Ddb$ and a constant $C\geq 1$ such that
\begin{equation}\label{1DL}
\norm{\phi(T)}\leq C\sup\bigl\{\vert\phi(z)\vert\, :\, z\in\Delta\bigr\},
\end{equation}
for all complex polynomials $\phi$ in one variable. This remarkable functional calculus property 
can be regarded as an improvement of von Neumann's inequality.

In this paper we construct a new unitary dilation property for
polynomially bounded, polygonal type operators $T\colon H\to H$,
see Section 3. Using this dilation property, we show in Section 4
that if $T_1,\ldots,T_d\colon H\to H$ are commuting contractions and if  
$T_1,\ldots,T_{d-2}\colon H\to H$ have a polygonal type, then
there exist a Hilbert space $H'$, two bounded operators $J\colon H\to H'$
and $Q\colon H'\to H$ and commuting unitaries $U_1,\ldots,U_d$ on $H'$ such that 
$$
T_1^{n_1} \cdots T_d^{n_d} = Q U_1^{n_1} \cdots U_d^{n_d} J, 
\qquad n_1,\ldots,n_d\geq 0.
$$
This implies that if $T_1,\ldots,T_d\colon H\to H$ are commuting contractions and  
$T_1,\ldots,T_{d-2}\colon H\to H$ have a polygonal type, then
$(T_1,\ldots,T_d)$ satisfies a generalized von Neumann inequality (\ref{1VN}). We also 
extend \cite{dL,FMI} by showing that $T_1,\ldots,T_d\colon H\to H$ are commuting polynomially bounded,
polygonal type operators, then they are jointly similar to contractions, that is,
there exists an invertible operator $S\colon H\to H$ such that $\norm{S^{-1}T_kS}\leq 1$
for all $k=1,\ldots,d$. Finally in Section 5, we extend (\ref{1DL}) to $d$-tuples. More precisely,
we show that $T_1,\ldots,T_d\colon H\to H$ are commuting polynomially bounded, polygonal type operators,
then there exist an open convex polygon 
$\Delta\subset\Ddb$ and a constant $C\geq 1$ such that
$$
\norm{\phi(T_1,\ldots,T_d))}\leq C\sup\bigl\{\vert\phi(z_1,\ldots,z_d)\vert\, :\, (z_1,\ldots,z_d)\in\Delta^d\bigr\},
$$
for all complex polynomials $\phi$ in $d$ variables. The approach to prove this 
functional calculus estimate is different from the one in \cite{dL}. It relies on a
Franks-McIntosh type decomposition established in \cite{BLM}.

Ritt operators are a sub-class of polygonal type operators  (for which we refer to  \cite{LM2}).
The results of the present paper, summarized in the previous paragraph, generalize results
from \cite[Section 5]{ArLM1} which concerns Ritt operators only.

\section{Polygonal type operators}\label{2Poly}
Throughout we consider a Hilbert space
$H$ and we let $B(H)$ denote the Banach algebra 
of all bounded operators on $H$, equipped with the operator norm. 
We let $I_H$ denote the identity operator on $H$.  
For any $T\in B(H)$, we let $\sigma(T)$ denote the spectrum of 
$T$ and for any $z$ in the resolvent set $\Cdb\setminus\sigma(T)$,
we let $R(z,T)=(z-T)^{-1}$ denote the resolvent operator at $z$.
As usual in this context, $z$ stands for $zI_H$. Also we let 
${\rm Ran}(T)$ and ${\rm Ker}(T)$ denote the range and
the kernel of $T$, respectively.

For simplicity we call ``polygon" any open convex 
polygon of $\Cdb$ (that is, a bounded, finite intersection of open complex half-planes).
All polygons considered in this paper will be included in the unit disc $\Ddb$.

\begin{definition}\label{2Pol-Type} 
We say that $T\in B(H)$ has a polygonal
type if there exist a polygon 
$\Delta\subset\Ddb$, with $\overline{\Delta}\cap\Tdb\neq\emptyset$,
and a constant $M\geq 0$
such that
$$
\sigma(T)\subset \overline{\Delta}
$$
and 
$$
\norm{R(z,T)}\leq M \max\bigl\{\vert z-\xi\vert^{-1}\, :\, \xi\in 
\overline{\Delta}\cap\Tdb\bigr\},\qquad
z\in\overline{\Ddb}^c.
$$
\end{definition}

As mentioned in the introduction, the class of Hilbert space
operators with a polygonal type first appeared in \cite{dL} and \cite{FMI},
although not with this terminology. 
They were recently studied in
\cite{BLM} (see also \cite{B}) in the Banach space setting, as generalisations
of Ritt operators.

Let $E\subset\Tdb$ be a non-empty finite set. Following \cite[Definition 2.1 and Remark 2.3]{BLM}, 
we say that
$T\in B(H)$ is a Ritt$_E$ operator if $\sigma(T)\subset \overline{\Ddb}$ and 
there exists a constant $M\geq 0$ such that 
$\norm{R(z,T)}\leq M\max\{\vert z-\xi\vert^{-1}\, :\, \xi\in E\}$ for all 
$z\in\overline{\Ddb}^c$. 
Following \cite[Definition 2.6]{BLM}, we let $E_r$ be the interior of the convex hull of $E$
and of the disc $\{\lambda\in\Cdb\, :\,\vert\lambda\vert <r\}$, for any $r\in (0,1)$. 
If $T$ is Ritt$_E$, then there exists $r\in(0,1)$ and a constant $M\geq 0$ such that 
$\sigma(T)\subset\overline{E_r}$,  by
\cite[Lemma 2.8]{BLM}. Furthermore for all $r\in(0,1)$,
there exists a polygon $\Delta$
such that
$$
E_r\subset \Delta\subset\Ddb.
$$
We easily deduce the following.

\begin{lemma}\label{RittE} Let $T\in B(H)$.
\begin{itemize}
\item [(1)] If $T$ is a Ritt$_E$ operator for some non-empty finite set $E\subset \Tdb$,
then $T$ has a polygonal type and we have $\sigma(T)\cap\Tdb\subset E$.
\item [(2)]   If $T$ has a polygonal type, then $T$ is a 
Ritt$_E$ operator for any non-empty finite subset $E\subset\Tdb$ 
containing  $\sigma(T)\cap\Tdb$.
\end{itemize}
\end{lemma}

We note the following result concerning sectoriality, established in \cite[Lemma 2.4]{BLM}.
We refer e.g. to \cite[Section 2.1]{Haa} for the
definition and basic properties of sectorial operators.

\begin{lemma}\label{2Sector}
If $T\in B(H)$ has a polygonal
type, then for any $\xi\in \sigma(T)\cap\Tdb$, the operator
$I_H-\overline{\xi}T$ is a sectorial operator of type $<\frac{\pi}{2}$.
\end{lemma}

Let $\P$ denote the algebra of all complex polynomials (in one variable). For any
non-empty bounded open set $\Omega\subset\Cdb$ and any $\phi\in\P$,
we set
$$
\norm{\phi}_{\infty,\Omega}=\sup\bigl\{\vert\phi(\lambda)\vert\, :\,
\lambda\in\Omega\bigr\}.
$$
Let us recall a few classical definitions and results.
An operator $T\in B(H)$ is called polynomially bounded if 
there exists a constant $C\geq 1$ such that 
$\norm{\phi(T)}\leq C\norm{\phi}_{\infty,{\mathbb D}}$ for all $\phi\in\P$.
Further $T\in B(H)$ is called similar to a contraction if there 
exists an invertible $S\in B(H)$ such that $S^{-1}TS$ is a contraction, i.e.
$\norm{S^{-1}TS}\leq 1$. Von Neumann's inequality implies that any
operator similar to a contraction is polynomially bounded. A
famous result of Pisier \cite{P1} asserts that the converse is wrong.
We refer to \cite[Chapter 9]{P0} for  more information
on this. Just after Pisier's result was circulated, De Laubenfels \cite{dL}
showed the following remarkable property (the paper \cite{FMI} also
contains the equivalence ``$(i)\,\Leftrightarrow\,(ii)$'').

\begin{theorem}\label{2DL} (\cite[Theorem 4.4]{dL})
Let $T\in B(H)$ be an operator of polygonal type. The following assertions
are equivalent.
\begin{itemize}
\item [(i)] $T$ is polynomially bounded.
\item [(ii)] $T$ is similar to a contraction.
\item [(iii)] There exists a polygon $\Delta\subset\Ddb$ and 
a constant $C\geq 1$ such that 
$$
\norm{\phi(T)}\leq C\norm{\phi}_{\infty,\Delta},
\qquad \phi\in\P.
$$
\end{itemize}
\end{theorem}

Note that an operator $T\in B(H)$ with a polygonal type is necessarily power bounded,
by \cite[Theorem 2.10]{BLM} and Lemma \ref{RittE}. However there exist operators of polygonal type 
on Hilbert space which are not polynomially bounded,
see \cite[Section 5]{LM}.

We now review a recent result of Bouabdillah \cite{B} on square functions
associated with operators of polygonal type. Bouabdillah's work is carried
out in the Banach space setting, however we will present it in the Hilbertian 
case only. We fix a finite set
$$
E=\{\xi_1,\ldots,\xi_N\}\subset\Tdb,
$$
where $N\geq 1$ and $\xi_1,\ldots,\xi_N$ are distinct.
For any Ritt$_E$ operator $T$, we define a square function associated with $T$
by setting, for any
$x\in H$,
$$
\norm{x}_T : = \biggl(\sum_{k=0}^\infty \bignorm{T^k
(I_H -\overline{\xi_1}T)^\frac12(I_H -\overline{\xi_2}T)^\frac12\cdots
(I_H -\overline{\xi_N}T)^\frac12x}^2\biggr)^\frac12.
$$
For all $j=1,\ldots,N$, the operator $I_H -\overline{\xi_j}T$
is sectorial (see above) hence the square root $(I_H -\overline{\xi_j}T)^\frac12$
is a well-defined element of $B(H)$ obtained 
by the usual functional calculus of sectorial operators, see e.g. \cite[Section 3]{Haa}.
We note that $\norm{x}_T$ may be infinite.

\begin{theorem}\label{2B} (\cite{B}) 
Let $T\in B(H)$ be a Ritt$_E$ operator. If $T$ is polynomially
bounded, then there exists a contant $C\geq 0$ such that 
$$
\norm{x}_T\leq C\norm{x}, \qquad x\in H.
$$
\end{theorem}

Here are a few comments on Theorem \ref{2B}. First, we note that 
if $T$ is polynomially bounded, then its adjoint $T^*$ is  polynomially bounded
as well. Hence Theorem \ref{2B} implies that in this case, we have an estimate 
$\norm{y}_{T^*}\leq C'\norm{y}$ for all $y\in H$. Second, it is shown in \cite{B}
that conversely, if there exist constants $C,C'\geq 0$ such that 
$\norm{x}_T\leq C\norm{x}$ and $\norm{y}_{T^*}\leq C'\norm{y}$ for all $x,y\in H$,
then $T$ is polynomially bounded. Third, \cite{B} considers other square functions than 
the above $\norm{\,\cdotp}_T$ and proves analogous results for them.

\section{A specific dilation}\label{3Dil}
Let $H$ be a Hilbert space.
If $T\in B(H)$ is a polynomially bounded operator 
of polygonal type, then there exists an invertible $S\in B(H)$
and a contraction $W\in B(H)$ such that $T=SWS^{-1}$, by Theorem \ref{2DL}.
According to  Nagy's dilation theorem (see e.g. \cite[Theorem 1.1]{P0}),
there exist a Hilbert space $K$, a unitary $V\in B(K)$ and an isometry $J_0\colon H\to K$
such that $W^n=J_0^*V^nJ_0$ for all $n\geq 0$. Set $J=J_0S^{-1}$ and 
$Q=SJ_0^*$. Then we obtain that
$$
T^n=QV^nJ,\qquad n\geq 0. 
$$
This is a unitary dilation of $T$.
The aim of this section is to devise a specific 
unitary dilation for polynomially bounded operators
of polygonal type, without any 
recourse to Nagy's dilation theorem.
This will be achieved in Proposition \ref{3Specific}.

For any index set $I$, we let  $\ell^2_{I}(H)$ denote the Hilbert space
of all $2$-summable families $(x_\iota)_{\iota\in I}$ of $H$. 
If $J\colon H\to \ell^2_{I}(H)$ is a bounded operator,
the operators $J_\iota\in B(H)$, $\iota\in I$, such that $J(x)=(J_\iota(x))_{\iota\in I}$ for all
$x\in H$, will be called the coordinates of $J$.

We fix a set 
$$
E=\{\xi_1,\ldots,\xi_N\}\subset\Tdb, 
$$
where $N\geq 1$ and $\xi_1,\ldots,\xi_N$ are distinct.
In the sequel,
we will use the Hilbertian direct sum $\ell^2_N(H)\mathop{\oplus}^2\ell^2_{\mathbb Z}(H)$.
We note that the latter can be
regarded as $\ell^2_{I}(H)$, where $I$ is the disjoint union  $I=\{1,\ldots,N\}\sqcup\Zdb$.
For any $T_1\in B(\ell^2_N(H))$ and $T_2\in B(\ell^2_{\mathbb Z}(H))$, we let 
$$
T_1\oplus T_2 \colon \ell^2_N(H)\mathop{\oplus}^2\ell^2_{\mathbb Z}(H)
\longrightarrow \ell^2_N(H)\mathop{\oplus}^2\ell^2_{\mathbb Z}(H)
$$
be the operator 
taking $(y_1,y_2)$ to $(T_1(y_1),T_2(y_2))$ for all $y_1\in \ell^2_N(H)$ and
$y_2\in \ell^2_{\mathbb Z}(H)$.

We let
$V_H\colon \ell^2_{\mathbb Z}(H)\to \ell^2_{\mathbb Z}(H)$
be the shift operator defined by
\begin{equation}\label{3VH}
V_H\bigl((x_k)_{k}\bigr) = (x_{k+1})_k,\qquad (x_k)_{k\in\mathbb Z}\in\ell^2_{\mathbb Z}(H).
\end{equation}
Then we let
$D_H\colon \ell^2_N(H)\to \ell^2_N(H)$ be the diagonal operator defined by
\begin{equation}\label{3DH}
D_H\bigl(x_1,x_2,\ldots,x_N\bigr) = \bigl(\xi_1 x_1,\xi_2 x_2,\ldots, \xi_N x_N\bigr),
\qquad x_1,\ldots,x_N\in H.
\end{equation}
Both $V_H$ and $D_H$ are unitaries.

\begin{proposition}\label{3Specific} 
Let $T\in B(H)$ be a polynomially bounded Ritt$_E$ operator.
Then there exist two bounded operators 
$$
J\colon H\longrightarrow \ell^2_N(H)\mathop{\oplus}^2\ell^2_{\mathbb Z}(H)
\qquad\hbox{and}\qquad
Q\colon \ell^2_N(H)\mathop{\oplus}^2\ell^2_{\mathbb Z}(H)\longrightarrow H
$$ 
such that 
$$
T^n=Q\bigl(D_H\oplus V_H^2\bigr)^nJ,\qquad n\geq 0, 
$$
and all the coordinates of $J$ belong to the bicommutant of $\{T\}$.
\end{proposition}

Note that $D_H\oplus V_H^2$ is a unitary, so 
this proposition provides a unitary dilation of $T$. We will 
prove it at the end of this section. We need
preparation lemmas.

Throughout the rest of this section, we let 
$(a_m)_{m\geq 0}$ denote the sequence of complex numbers provided by the Taylor expansion,
\begin{equation}\label{expansion equation}
\frac{1}{\prod_{j=1}^{N}(1-\overline{\xi_j}z)} = 
\sum_{m=0}^{\infty}a_{m}z^m,\qquad z\in\Ddb.
\end{equation}

\begin{lemma}\label{Boundedness of ak's}
Let $(a_m)_{m\geq 0}$ be defined by (\ref{expansion equation}).
\begin{itemize}
\item [(1)] There exist $\beta_1,\ldots,\beta_N\in\Cdb$ such that $a_m=  \sum_{i=1}^{N}\beta_i(\overline{\xi_i})^m$ for all $m\geq 0$.
\item [(2)] The sequence $(a_m)_{m\geq 0}$ is bounded.
\end{itemize}
\end{lemma}

\begin{proof}
Property (2) clearly follows from property (1), so we only need to prove (1).  We use
induction on $N$. For convenience we write $a_{m,N}$ instead of $a_m$ for the coefficients
in (\ref{expansion equation}). In the case $N=1$, we have
$$
\frac{1}{1-\overline{\xi_1}z} = \sum_{m=0}^{\infty}\overline{\xi_1}^mz^m,\qquad z\in\Ddb,
$$
hence the result follows at once. Now let us assume that $N\geq 2$ and that property (1) holds true for $N-1$. 
Thus we have an expansion
$$
\frac{1}{\prod_{j=1}^{N-1}(1-\overline{\xi_j}z)} = \sum_{m=0}^{\infty}a_{m,N-1}z^m,\qquad z\in\Ddb,
$$
where $a_{m,N-1} = \sum_{i=1}^{N-1}\alpha_i(\overline{\xi_i})^m$ for some constants 	
$\alpha_1,\ldots,\alpha_{N-1}\in\mathbb{C}.$
We write
$$
\sum_{m=0}^{\infty}a_{m,N}z^m = \frac{1}{1-\overline{\xi_N}z}\,\biggl(\sum_{j=0}^{\infty}a_{j,N-1}z^j\biggr)
= \biggl(\sum_{k=0}^\infty \overline{\xi_N}^k z^k\biggr)\biggl(\sum_{j=0}^{\infty}a_{j,N-1}z^j\biggr).
$$
This implies, for all $m\geq 0$,
\begin{align*}
a_{m,N}&=\sum_{j=0}^{m}a_{j,N-1}\overline{\xi_N}^{m-j}\\
&= \sum_{j=0}^{m}\Bigl(\sum_{i=1}^{N-1}\alpha_i(\overline{\xi_i})^j\Bigr)(\overline{\xi_{N}})^{m-j}\\
&=(\overline{\xi_{N}})^m\sum_{j=0}^{m}\sum_{i=1}^{N-1}\alpha_i(\overline{\xi_i})^j(\xi_{N})^j\\
&=(\overline{\xi_{N}})^m\sum_{i=1}^{N-1}\alpha_i\biggl[\frac{1-(\overline{\xi_i}\xi_{N})^{m+1}}{1-\overline{\xi_i}\xi_{N}}\biggr]\\
&=\sum_{i=1}^{N-1}\alpha_i\frac{(\overline{\xi_{N}})^m}{1-\overline{\xi_i}\xi_{N}}
-\sum_{i=1}^{N-1}\alpha_i\frac{\overline{\xi_i}\xi_N}{1-\overline{\xi_i}\xi_{N}}(\overline{\xi_i})^m.
\end{align*}
Define
$$
\beta_i = - \frac{\alpha_i\overline{\xi_i}\xi_{N}}{1-\overline{\xi_i}\xi_{N}}
\quad\hbox{for}\ 1\leq i\leq N-1\qquad\hbox{and}\qquad
\beta_{N} = \sum_{i=1}^{N-1}\frac{\alpha_i}{1-\overline{\xi_i}\xi_{N}}.
$$
Then we have
$a_{m,N}=\sum_{i=1}^{N}\beta_i(\overline{\xi_i})^m$,
which  proves (1). 
\end{proof}

\begin{lemma}\label{elements in the range}
Let $T\in B(H)$ be a Ritt$_E$ operator, and assume that $T$ is polynomially bounded. 
Then for any $x\in \overline{\rm Ran}\Bigl(\prod_{j=1}^{N}(I_H-\overline{\xi_j}T)\Bigr)$,	
we have
$$
\sum_{m=0}^{\infty}a_{m}T^m\prod_{j=1}^{N}(I_H-\overline{\xi}_jT)x=x.
$$
(The convergence of the series is part of the result.)
\end{lemma}

\begin{proof}
Consider the polynomial expansion
$$
\prod_{j=1}^{N}(1-\overline{\xi_j}z) = \sum_{i=0}^{N}c_iz^i,\qquad z\in \Ddb,
$$
for some $c_0,c_1,\ldots,c_N\in\Cdb$. For convenience, we set $c_i=0$ if $i>N$.
Observe that 
$$
\Bigl( \sum_{i=0}^{\infty}c_iz^i\Bigr)\Bigl(\sum_{m=0}^{\infty}a_{m}z^m\Bigr)=
\prod_{j=1}^{N}(1-\overline{\xi_j}z)\sum_{m=0}^{\infty}a_{m}z^m = 1,
\qquad z\in\Ddb.
$$
This implies that $c_0a_0=1$ and that for any integer $r\geq 1$,  
$$
\sum_{i+m=r} c_ia_m=0.
$$
For any integer $k\geq N$, we let $S_k\in\P$ be defined by 
$$
S_k(z) = \Bigl( \sum_{i=0}^{N}c_i z^i\Bigr)\Bigl(\sum_{m=0}^{k}a_{m}z^m\Bigr),
$$
and we write it as 
$$
S_k(z) =\sum_{r=0}^{k+N}\gamma_{r,k} z^r.
$$
We note that if $r\leq k$, then $\gamma_{r,k} = \sum_{i+m=r} c_ia_m$. 
Hence $\gamma_{0,k}=1$ and
$\gamma_{r,k}=0$ if $1\leq r\leq k$.  Thus we actually have
$$
S_k(z) = 1 +\sum_{r=k+1}^{k+N}\gamma_{r,k} z^r,
\qquad\hbox{with} \quad  \gamma_{r,k} = \sum_{m=r-N}^{k}c_{r-m}  a_{m}\quad\hbox{for} \ k+1\leq r\leq k+N.
$$
Since $(a_m)_{m\geq 1}$ is a bounded sequence, by Lemma \ref{Boundedness of ak's}, the above expression of $\gamma_{r,k}$
implies that
\begin{equation}\label{3Bounded}
\bigl\{\gamma_{r,k}\, : \,k\geq N,\  k+1\leq r\leq k+N\bigr\}\quad\hbox{is bounded.}
\end{equation} 
Consequently, 
the set $\{S_k(z)\, :\, k\geq N,\, z\in\Ddb\}$ is bounded. Since $T$ is polynomially bounded, this implies that
\begin{equation}\label{3Uniform}
\sup_{k\geq N}\norm{S_k(T)}\,<\infty.
\end{equation}

Let $x\in {\rm Ran}\Bigl(\prod_{j=1}^{N}(I_H-\overline{\xi_j}T)\Bigr)$ and write 
$x=\prod_{j=1}^{N}(I_H-\overline{\xi_j}T)y$
for some $y\in H$. Then for any $k\geq N$, we have
$$
S_k(T)x = \Bigl(I_H+\sum_{r=k+1}^{k+N}\gamma_{r,k}T^r\Bigr)x = 
x +\sum_{r=k+1}^{k+N}\gamma_{r,k}\Bigl(T^r\prod_{j=1}^{N}(I_H-\overline{\xi_j}T)\Bigr)y.
$$
Since $T$ is polynomially bounded, we have 
$$
\sum_{r=0}^\infty \bignorm{T^r
(I_H -\overline{\xi_1}T)^\frac12(I_H -\overline{\xi_2}T)^\frac12\cdots
(I_H -\overline{\xi_N}T)^\frac12 y}^2\,<\infty,
$$
by Theorem  \ref{2B}.
Thus, $\bignorm{T^r(I_H -\overline{\xi_1}T)^\frac12(I_H -\overline{\xi_2}T)^\frac12\cdots
(I_H -\overline{\xi_N}T)^\frac12 y} \to 0$, and hence
$$
\Bigl(T^r\prod_{j=1}^{N}(I_H-\overline{\xi_j}T)\Bigr)y
\longrightarrow 0,
$$
when $r\to\infty$. Using (\ref{3Bounded}) again, we deduce that $S_k(T)x\to x$ when $k\to\infty$.

Now assume that $x\in \overline{\rm Ran}\Bigl(\prod_{j=1}^{N}(I_H-\overline{\xi_j}T)\Bigr)$. It follows from above and from 
the uniform estimate (\ref{3Uniform}) that  $S_k(T)x\to x$ when $k\to\infty$.
Since 
$$
S_k(T)= \sum_{m=0}^{k}a_{m}T^m\prod_{j=1}^{N}(I_H-\overline{\xi}_jT),
$$
this yields the result.
\end{proof}

\begin{lemma}\label{decomposition}
Let $T\in B(H)$ be any power bounded operator. Then $H$ admits a direct sum decomposition
\begin{equation}\label{Dec}
H=\overset{N}{\underset{j=1}{\oplus}}{\rm Ker}(I_H-\overline{\xi_j}T)
\oplus\overline{\rm Ran}\Bigl(\prod_{j=1}^{N}(I_H-\overline{\xi_j}T)\Bigr).
\end{equation}
Furthermore, the $(N+1)$ projections associated with this decomposition 
all belong to the bicommutant of $\{T\}$.
\end{lemma}

\begin{proof}
The operator $\overline{\xi_1}T$ is power bounded. Since $H$ is reflexive, the 
Mean ergodic theorem (see e.g. \cite[Subsection 2.2.1]{Kre}) asserts that we have a 
direct sum decomposition
$$
H={\rm Ker}(I_H-\overline{\xi_1}T)\oplus\overline{\rm Ran}(I_H-\overline{\xi_1}T).
$$
Moreover, if we let $P\in B(H)$ be the projection with range equal to  
${\rm Ker}(I_H-\overline{\xi_1}T)$ and kernel equal to  
$\overline{\rm Ran}(I_H-\overline{\xi_1}T)$, we have
$$
P(x) = \lim_{n\to\infty}\frac{1}{n+1}\sum_{k=0}^n(\overline{\xi_1}T)^kx,
\qquad x\in H.
$$
Hence $P$ belongs to the bicommutant of $\{T\}$.

Consider the restriction $T_1=T|_{\overline{\rm Ran}(I_H-\overline{\xi_1}T)} \in B\bigl(\overline{\rm Ran}(I_H-\overline{\xi_1}T) \bigr)$. 
Now applying  the Mean ergodic theorem to $T_1$,  
we obtain a direct sum decomposition
$$
\overline{\rm Ran}(I_H-\overline{\xi_1}T) = \bigl({\rm Ker}(I_H-\overline{\xi_2}T)\cap
\overline{\rm Ran}(I_H-\overline{\xi_1}T)\bigr)
\oplus\overline{\rm Ran}\bigl((I_H-\overline{\xi_2}T)(I_H-\overline{\xi_1}T)\bigr).
$$
Moreover the  projections associated with this  decomposition belong to
the bicommutant of $T_1$.

Let  $x\in{\rm Ker}(I_H-\overline{\xi_2}T)$. Then $x=\overline{\xi_2}Tx$ and hence 
$$
(I-\overline{\xi_1}T)x = \overline{\xi_2}Tx-\overline{\xi_1}Tx
=(1-\xi_2\overline{\xi_1})\overline{\xi_2}Tx
=(1-\xi_2\overline{\xi_1})x.
$$
Since $\xi_2\not=\xi_1$, we have $1-\xi_2\overline{\xi_1}\not=0$
and we obtain that 
$$
x=\frac{(I-\overline{\xi_1}T)}{(1-\xi_2\overline{\xi_1})}x. 
$$
This implies that
$x\in \overline{\rm Ran}(I_H-\overline{\xi_1}T).$ Consequently,
${\rm Ker}(I_H-\overline{\xi_2}T)\cap
\overline{\rm Ran}(I_H-\overline{\xi_1}T)$ is equal to ${\rm Ker}(I_H-\overline{\xi_2}T)$,
and hence
$$
H={\rm Ker}(I_H-\overline{\xi_1}T)\oplus{\rm Ker}(I_H-\overline{\xi_2}T)\oplus
\overline{\rm Ran}\bigl((I_H-\overline{\xi_2}T)(I_H-\overline{\xi_1}T)\bigr).
$$
Moreover the projections associated with this direct sum decomposition belong to
the bicommutant of $T$.

Since $\xi_i\neq\xi_j$ for any $1\leq i\neq j\leq N$, the result now follows by 
induction.
\end{proof}

\begin{proof}[Proof of Proposition \ref{3Specific}]
We set 
$$
A= \prod_{j=1}^{N}(I_H-\overline{\xi_j}T)^\frac{1}{2}.
$$
Since $T$ is polynomially bounded, it satisfies a square function estimate
$$
\sum_{k=0}^\infty \bignorm{T^k
(I_H -\overline{\xi_1}T)^\frac12(I_H -\overline{\xi_2}T)^\frac12\cdots
(I_H -\overline{\xi_N}T)^\frac12x}^2 \leq C^2||x||^2,\qquad x\in H,
$$
by Theorem \ref{2B}.
According to Lemma \ref{decomposition}, any 
$x\in H$ can be uniquely written as 
$$
x=x_1+x_2+\cdots+x_N + x_{N+1},
$$
where $x_j\in {\rm Ker}(I_H-\overline{\xi_j}T)$ for $1\leq j\leq N$ and 
$x_{N+1}\in 
\overline{\rm Ran}\Bigl(\prod_{j=1}^{N}(I_H-\overline{\xi_j}T)\Bigr)$. Applying 
the above square function estimate to $x_{N+1}$, we may therefore define 
a bounded operator
$$
J\colon H\longrightarrow \ell^2_N(H)\mathop{\oplus}^2\ell^2_{\mathbb Z}(H)
$$
by 
\begin{align*}
J(x)=(x_1,\ldots,x_N) + \bigl(\ldots, 0\ldots, 0, Ax_{N+1}&,  TAx_{N+1}, TAx_{N+1},
\ldots, \\ & \ T^kAx_{N+1}, T^kAx_{N+1}, T^{k+1}Ax_{N+1},\ldots\bigr),
\end{align*}
In the definition of $J$, the term  $T^kAx_{N+1}$ appears at the $2k$-th position and 
the term $T^{k+1}Ax_{N+1}$ appears at the $(2k+1)$-th position of the sequence, for every $k\geq 0$.

Similarly since $T^*$ is polynomially bounded, it also satisfies a square function estimate
$$
\sum_{k=0}^\infty \bignorm{T^{*k}
(I_H - \xi_1 T^*)^\frac12 (I_H - \xi_2T^*)^\frac12\cdots
(I_H - \xi_N T^*)^\frac12 y}^2  \leq {C'}^2||y||^2,\qquad y\in H.
$$
Moreover Lemma \ref{decomposition} ensures that we have a
direct sum decomposition
$$
H=\overset{N}{\underset{j=1}{\oplus}} {\rm Ker}
(I_H-\xi_jT^*)\oplus\overline{\rm Ran}\Bigl(\prod_{j=1}^{N}(I_H-{\xi_j}T^*)\Bigr).
$$
For any $y\in H$, let 
$$
y=y_1+y_2+\cdots +y_N+ y_{N+1},
$$
with $y_j\in {\rm Ker}(I_H- \xi_jT^*)$ for $1\leq j\leq N$ and 
$y_{N+1}\in \overline{\rm Ran}\Bigl(\prod_{j=1}^{N}(I- \xi_j T^*)\Bigr)$, and 
define a bounded linear operator 
$$
\widetilde{J}\colon
H\longrightarrow \ell^2_N(H)\mathop{\oplus}^2\ell^2_{\mathbb Z}(H)
$$
by 
\begin{align*}
\widetilde{J}(y)
= (y_1,\ldots, y_N) + \bigl(\ldots, 0\ldots, 0, \overline{a_0} A^*y_{N+1}&,  
\overline{a_1} A^*y_{N+1}, \ldots, \\ & \ \overline{a_{2k-1}} T^{*(k-1)}A^*y_{N+1}, 
\overline{a_{2k}} T^{*k}A^*y_{N+1}, 
\ldots\bigr),
\end{align*} 
where the cefficients $a_k$ are given by (\ref{expansion equation}). In this definition,
$\overline{a_{2k-1}} T^{*(k-1)}A^*y_{N+1}$ appears at the 
$(2k-1)$-{th} position for all $k\geq 1$
whereas
$\overline{a_{2k}} T^{*k}A^*y_{N+1}$ appears at the 
$2k$-{th} position for all $k\geq0$.

We set $Q={\widetilde{J}}^*\colon \ell^2_N(H)\mathop{\oplus}^2\ell^2_{\mathbb Z}(H)\to H$ and  
$$
V=D_H\oplus V_H^2\colon \ell^2_N(H)\mathop{\oplus}^2\ell^2_{\mathbb Z}(H)
\longrightarrow
 \ell^2_N(H)\mathop{\oplus}^2\ell^2_{\mathbb Z}(H).
$$
For any integer $n\geq 0$, we have
\begin{align*}
V^nJ(x)=(\xi_1^nx_1,\xi_2^n x_2,\ldots,\xi_N^n x_N) + \bigl(\ldots, *&,\ldots, *, 
T^nAx_{N+1},  T^{n+1}Ax_{N+1}, T^{n+1}Ax_{N+1},
\ldots, \\ & \ T^{n+k}Ax_{N+1}, T^{n+k}Ax_{N+1}, T^{n+k+1}Ax_{N+1},\ldots\bigr),
\end{align*}
where the $*$ stand for terms at negative positions and for every $k\geq 0$,
$T^{n+k}Ax_{N+1}$ appears at the $2k$-th position and $T^{n+k+1}Ax_{N+1}$ appears at the $(2k+1)$-th position, for every $k\geq 0$.

Combining with the defintion of $\widetilde{J}$, we have
\begin{align*}
\langle QV^nJ(x),y\rangle &  = 
\langle V^nJ(x),\widetilde{J}(y)\rangle \\ & =
\sum_{j=1}^N\xi_j^n \langle x_j,y_j\rangle \\
& +\Bigl\langle 
\bigl(T^nAx_{N+1},  T^{n+1}Ax_{N+1}, T^{n+1}Ax_{N+1},
\ldots, \ T^{n+k}Ax_{N+1}, T^{n+k}Ax_{N+1}, \ldots\bigr),\\ &
\qquad \bigl( \overline{a_0} A^*y_{N+1},  
\overline{a_1} A^*y_{N+1}, \ldots,  \ \overline{a_{2k-1}} T^{*(k-1)}A^*y_{N+1}, 
\overline{a_{2k}} T^{*k}A^*y_{N+1}, 
\ldots\bigr)\Bigr\rangle\\ 
 & = 
\sum_{j=1}^N\xi_j^n \langle x_j,y_j\rangle\\
& +\sum_{k=0}^\infty\Bigl(a_{2k}\bigl\langle
T^{n+k}Ax_{N+1}, T^{*k}A^*y_{N+1}\bigr\rangle
+ a_{2k+1}\bigl\langle
T^{n+k+1}Ax_{N+1}, T^{*k}A^*y_{N+1}\bigr\rangle\Bigr).
\end{align*}
This yields
$$
\langle QV^nJ(x),y\rangle   = 
\sum_{j=1}^N\xi_j^n \langle x_j,y_j\rangle+
\sum_{m=0}^\infty a_m \bigl\langle T^{m+n}A^2
x_{N+1},y_{N+1}\bigr\rangle.
$$
Since $x_{N+1}\in \overline{\rm Ran}\Bigl(\prod_{j=1}^{N}(I_H-\overline{\xi_j}T)\Bigr)$ 
we deduce, using Lemma 
\ref{elements in the range}, that
$$
\langle QV^nJ(x),y\rangle   = 
\sum_{j=1}^N\xi_j^n \langle x_j,y_j\rangle+ \langle T^nx_{N+1},y_{N+1}\rangle.
$$
For any $1\leq j\leq N$, $x_j\in {\rm Ker}(I_H-\overline{\xi_j}T)$,
hence $Tx_j = \xi_jx_j$. Thus, $T^nx_j = {\xi_j}^nx_j$ for all $n\geq 0.$
Therefore, we have proved that
$$
\langle QV^nJ(x),y\rangle   = 
\sum_{j=1}^N   \langle T^n x_j,y_j\rangle+ \langle T^nx_{N+1},y_{N+1}\rangle.
$$

It follows from the proof of Lemma \ref{decomposition} that 
$$
\overline{\rm Ran}(I_H-\overline{\xi_1}T)
= \overset{N}{\underset{j=2}{\oplus}}{\rm Ker}(I_H-\overline{\xi_j}T)
\oplus\overline{\rm Ran}\Bigl(\prod_{j=1}^{N}(I_H-\overline{\xi_j}T)\Bigr).
$$
Since $\overline{\rm Ran}(I_H-\overline{\xi_1}T)={\rm Ker}(I_H- \xi_1T^*)^\perp$, this 
implies that $\langle x_j,y_1\rangle=0$ for all $j\geq 2$. Similarly, we have
$\langle x_j,y_i\rangle=0$ for all $1\leq j\not= i\leq N+1$. 

For any $n\geq 0$, $T^n(x) = T^n(x_1)+T^n(x_2)+\cdots+T^n(x_{N+1})$ is the decomposition
of $T^n(x)$ corresponding to the direct sum (\ref{Dec}). Hence we also have
$\langle T^n x_j,y_i\rangle=0$ for all $1\leq j\not= i\leq N+1$.
Consequently,
$$
\sum_{j=1}^N   \langle T^n x_j,y_j\rangle+ \langle T^nx_{N+1},y_{N+1}\rangle
=\langle T^nx,y\rangle.
$$
Thus, the above calculation yields $T^n = QV^nJ.$

The coordinates of $J$ are either the projections 
$x\mapsto x_j$, for $1\leq j\leq N$, the zero operator or the operators
$x\mapsto  T^kAx_{N+1}$, for $k\geq 0$. They all belong to the bicommutant of
$\{T\}$ by the last sentence of Lemma \ref{decomposition}.
\end{proof}

\section{Generalized von Neumann inequalities and joint similarities}\label{4vN} 
In this section we deal with finite commuting families of operators on $H$.
We will extend most of the results of \cite[Section 5]{ArLM1} to 
operators of polygonal type.

We aim at writing a unitary 
dilation for a $d$-tuple $(T_1,\ldots,T_d)$ of commuting elements
of $B(H)$, of which at least $(d-2)$ of them have a
poygonal type and are polynomially
bounded. For this purpose we will appeal to \cite[Lemma 4.1]{ArLM1}, that we state
(and explain) in 
the Hilbertian setting for convenience.

Given a second Hilbert space $K$, we let $K\overset{2}{\otimes} H$
denote the Hilbertian tensor product of $K$ and $H$. We recall that
for any $T\in B(H)$ and $A\in B(K)$, the operator 
$A\otimes T\colon K\otimes H\to K\otimes H$ uniquely extends
to a bounded operator
$$
A\overline{\otimes}T\colon
K\overset{2}{\otimes} H\longrightarrow K\overset{2}{\otimes} H, 
$$
with $\norm{A\overline{\otimes}T}=\norm{A}\norm{T}$. Furthermore,
$A\overline{\otimes}T$ is a unitary if $A$ and $T$ are unitaries.

In the sequel, for any integer $j\geq 1$, we set $K^{\otimes j}=
K\overset{2}{\otimes}\cdots \overset{2}{\otimes} K$ ($j$ copies).

\begin{lemma}\label{4JointDil}
Let $1 \leq m < d$ be integers. Let $T_1,\ldots,T_d\in B(H)$ 
be commuting operators. 
Assume that there exist two Hilbert spaces $K,L$ for
which the following properties hold:

\begin{itemize}
\item[(a)] For every $k=1,\ldots,m$, there exist a unitary $V_k\in B(K)$  
and two bounded operators $J_k \colon H \to K\overset{2}{\otimes} H$ 
and $Q_k \colon K\overset{2}{\otimes} H \to H$ such that
$$
T_k^{n_k} = Q_k (V_k \overline{\otimes} I_H)^{n_k} J_k, \qquad n_k\geq 0.
$$
\item[(b)] There exist  two bounded operators 
$J_{m+1} \colon H \to L$ and $Q_{m+1} \colon L \to H$ 
and commuting unitaries $V_{m+1},\ldots,V_d \in B(L)$ such that
$$
T_{m+1}^{n_{m+1}} \cdots T_d^{n_d} = Q_{m+1} V_{m+1}^{n_{m+1}} 
\cdots V_d^{n_d} J_{m+1}, \qquad n_{m+1},\ldots,n_d\geq 0.
$$
\item[(c)] For every $i=1,\ldots,m$ and $j=1,\ldots,d$, we have
$$
J_i T_j = (I_{K} \overline{\otimes} T_j) J_i.
$$
\end{itemize}
Then there exist commuting unitaries $U_1,\ldots,U_d$ on $K^{\otimes m}
\overset{2}{\otimes} L$, as well as
two bounded operators $J \colon H\to K^{\otimes m}
\overset{2}{\otimes} L$ and $Q \colon K^{\otimes m}
\overset{2}{\otimes} L \to  H$ such that
\begin{equation} \label{2combdil}
T_1^{n_1} \cdots T_d^{n_d} = Q U_1^{n_1}\cdots U_d^{n_d} J, 
\qquad n_1,\ldots,n_d\geq 0.
\end{equation} 
\end{lemma}

\begin{proof}
For any $k=1,\ldots,m$,  define $U_k$ on $K^{\otimes m}
\overset{2}{\otimes} L$ by
$$
U_k = I_{K^{\otimes (k-1)}}\overline{\otimes} V_k\overline{\otimes}
I_{K^{\otimes (m-k)}}\overline{\otimes}I_L.
$$
Next for $k=m+1,\ldots,d$,  define $U_k$ on $K^{\otimes m}
\overset{2}{\otimes} L$ by
$$
U_k = I_{K^{\otimes m}}\overline{\otimes} V_k.
$$
Then $U_1,\ldots,U_d$ are commuting unitaries.

Next define  $J \colon H\to K^{\otimes m}
\overset{2}{\otimes} L$ and $Q \colon K^{\otimes m}
\overset{2}{\otimes} L \to  H$ by
$$
J= (I_{K^{\otimes m}}\overline{\otimes} J_{m+1})
(I_{K^{\otimes (m-1)}}\overline{\otimes} J_m)\cdots
(I_{K^{\otimes 2}}\overline{\otimes} J_3)
(I_{K}\overline{\otimes} J_2)J_1
$$
and
$$
Q=Q_1(I_K\overline{\otimes} Q_2)
(I_{K^{\otimes 2}}\overline{\otimes} Q_3)\cdots
(I_{K^{\otimes (m-1)}}\overline{\otimes} Q_m)
(I_{K^{\otimes m}}\overline{\otimes} Q_{m+1}).
$$
Then the proof of  \cite[Lemma 4.1]{ArLM1} shows that the operators
$J,Q,U_1,\ldots,U_d$
satisfy the dilation property (\ref{2combdil}).
\end{proof}

Let $d\geq 1$ be an integer. Extending  definitions from Section
\ref{2Poly}, we let $\P_d$ denote the algebra of all complex polynomials in $d$ variables. 
Further given any
non-empty bounded open set $\Omega\subset\Cdb^d$ and any $\phi\in\P_d$,
we set
$\norm{\phi}_{\infty,\Omega}=\sup\bigl\{\vert\phi(\lambda)\vert\, :\,
\lambda\in\Omega\bigr\}$.

Let $(T_1,\ldots,T_d)$ be a $d$-tuple of commuting operators on $H$. We say that 
$(T_1,\ldots,T_d)$ is similar to a $d$-tuple of contractions if  there 
exists an invertible $S\in B(H)$ such that $S^{-1}T_k S$ is a contraction
for all $k=1,\ldots, d$. Note that there exist couples $(T_1,T_2)$ of commuting 
opertors such that $T_1$ and $T_2$ are each similar to a contraction 
but $(T_1,T_2)$ is not similar to a couple of contractions. That is,
there exist invertibles $S_1,S_2\in B(H)$ such that 
\begin{equation}\label{4Couple}
\norm{S_1^{-1}T_1 S_1}\leq 1
\qquad\hbox{and}\qquad
\norm{S_2^{-1}T_2 S_2}\leq 1,
\end{equation}
but (\ref{4Couple}) cannot be achieved with $S_2=S_1$. We refer to
\cite{P2} for an example.

\begin{theorem}\label{4UnitDil}
Assume that $d\geq 3$ and let $T_1,\ldots,T_d\in B(H)$ 
be commuting operators such that:
\begin{itemize}
\item[(a)] For every $k=1,\ldots,d-2$, $T_k$ is polynomially bounded
with a polygonal type;
\item[(b)] The couple $(T_{d-1},T_d)$ is similar to a couple of contractions. 
\end{itemize}
Then the following two properties hold.
\begin{itemize}
\item [(1)]
There exist a Hilbert space $H'$, two bounded operators $J\colon H\to H'$
and $Q\colon H'\to H$ and commuting unitaries $U_1,\ldots,U_d$ on $H'$ such that 
$$
T_1^{n_1} \cdots T_d^{n_d} = Q U_1^{n_1} \cdots U_d^{n_d} J, 
\qquad n_1,\ldots,n_d\geq 0.
$$
\item [(2)]
There exists a constant $C\geq 1$ such that
$$
\norm{\phi(T_1,\ldots,T_d)}\leq C\norm{\phi}_{\infty,{\mathbb D}^d},
\qquad \phi\in\P_d.
$$
\end{itemize}
\end{theorem}

\begin{proof}
To prove (1),  we will show that $(T_1,\ldots,T_d)$ satisfies the assumptions of Lemma \ref{4JointDil}
for $m=d-2$.  Since $(T_{d-1},T_d)$ is similar to a couple of contractions,
there exists an invertible operator $S\in B(H)$ such that $S^{-1}T_{d-1}S$ and $S^{-1}T_dS$ are contractions. 
These contractions commute. Hence by Ando's dilation theorem (see \cite{A}
or \cite[Theorem 1.2]{P0}),
there exist a Hilbert space $L$,  an isometry $J_0\colon H\to L$ and commuting unitaries
$V_{d-1}, V_d$ on $L$ such that 
$$
(S^{-1}T_{d-1}S)^{n_{d-1}}(S^{-1}T_dS)^{n_d} = J_0^* (V_{d-1}^{n_{d-1}}V_d^{n_d})J_0,
\qquad n_{d-1},n_d\geq 0.
$$
This readily implies that $(T_{d-1},T_d)$ satisfies the assumption (b) of Lemma \ref{4JointDil},
with 
$$
J_{d-1}=J_0S^{-1}\qquad\hbox{and}\qquad Q_{d-1} =  SJ_0^*.
$$

Applying Lemma \ref{RittE}, we let $E=\{\xi_1,\ldots,\xi_N\}\subset \Tdb$ such that
$T_k$ is a Ritt$_E$ operator for every $1\leq k\leq d-2$. We consider $V_H$ and $D_H$ defined by
(\ref{3VH}) and (\ref{3DH}) and we introduce their scalar versions
$V=V_{\mathbb C}\in B(\ell^2_{\mathbb Z})$ and $D=D_{\mathbb C}\in B(\ell^2_N)$.

For any index set $I$, we have a natural identification $\ell^2_I(H)\simeq  
\ell^2_I\overset{2}{\otimes} H$. Hence we may write
$$
\ell^2_N(H)\mathop{\oplus}^2\ell^2_{\mathbb Z}(H) = K
\overset{2}{\otimes} H,
\qquad\hbox{with}\ K=\ell^2_N\mathop{\oplus}^2\ell^2_{\mathbb Z}.
$$
Under this identification, we have
$$
D_H\oplus V_H^2 = (D\oplus V^2)\overline{\otimes} I_H.
$$
Applying Proposition \ref{3Specific} with every  $1\leq k\leq d-2$, we find
bounded operators $J_k\colon H\to K\overset{2}{\otimes} H$
and $Q_k\colon  K\overset{2}{\otimes} H\to H$ such that
$(T_1,\ldots, T_{d-2})$ satisfies the assumption (a) of Lemma \ref{4JointDil},
with $V_k =D\oplus V^2$, and the coordinates of $J_k$ belong to the bicommutant of
$\{T_k\}$.
Since $T_1,\ldots, T_{d}$ are commuting, the latter property  ensures that
$(T_1,\ldots,T_d)$ satisfies the assumption (c) of Lemma \ref{4JointDil}.
This completes the proof of (1).

Using (1), we may write
$$
\phi(T_1,T_2,\ldots,T_d) = Q\phi(U_1,U_2,\ldots, U_d)J
$$
for all $\phi\in\P_d$.  This implies 
$$
\norm{\phi(T_1,T_2,\ldots,T_d)}\leq\norm{J}\norm{Q}\norm{\phi(U_1,U_2,\ldots, U_d)}.
$$
Hence
$||\phi(T_1,T_2,\ldots,T_d)||\leq C||\phi||_{\infty,\mathbb {D}^d},$ with $C=\norm{J}\norm{Q}$, for all polynomial 
$\phi\in\P_d$.
\end{proof}

The following corollary asserts that a von Neumann type inequality for
a commuting $d$-tuple of 
contractions holds if at least $d-2$ of these contractions have a polygonal type.
This is  a straightforward consequence of Theorem \ref{4UnitDil}.

\begin{corollary}\label{4Contr}
Assume that $d\geq 3$ and let $T_1,\ldots,T_d\in B(H)$ 
be commuting contractions 
such that $T_1,\ldots, T_{d-2}$ have
a polygonal type. Then there exists a constant $C\geq 1$ such that
\begin{equation}\label{4MVN}
\norm{\phi(T_1,\ldots,T_d)}\leq C\norm{\phi}_{\infty,{\mathbb D}^d},
\qquad \phi\in\P_d.
\end{equation}
\end{corollary}

The following is a multivariable analogue of the equivalence 
``$(i)\Leftrightarrow(ii)$" of Theorem \ref{2DL}.

\begin{corollary}\label{4VNI}
Let $H$ be a Hilbert space, let $d\geq 1$ and let $T_1,\ldots,T_d\in B(H)$ 
be commuting, polynomially bounded operators with
a polygonal type. Then $(T_1,\ldots,T_d)$ is similar to a $d$-tuple 
of contractions.
\end{corollary}

\begin{proof} 
Let us apply Theorem \ref{4UnitDil} to the
$(d+2)$-tuple  $(T_1,T_2,\ldots,T_d,I_H,I_H)$. 
For some $Q, J$  and commuting unitaries $U_1,U_2,\ldots,U_d$, we may write
$$
T_1^{n_1}T_2^{n_2}\ldots T_d^{n_d} = QU_1^{n_1}U_2^{n_2}\ldots U_d^{n_d}J,
\qquad n_1,\ldots,n_d\geq 0.
$$ 
Then the argument in Part (3) of \cite[Theorem 5.1]{ArLM1} shows the existence of an
invertible operator $S\in B(H)$ such that $S^{-1}T_jS$ is a contraction for any $j=1,2,\ldots,d$.
We skip the details.
\end{proof}

\section{A multivariable poygonal functional calculus}\label{FC}
The aim of this section is to improve the inequality (\ref{4MVN})
in the case when $T_1,\ldots,T_d$ are all of polygonal type. This will  
provide  a  multivariable analogue of the equivalence 
``$(i)\Leftrightarrow(iii)$" of Theorem \ref{2DL}.

\begin{theorem}\label{5Poly-FC}
Let $d\geq 1$ and let $T_1,\ldots,T_d\in B(H)$ 
be commuting, polynomially bounded operators with
a polygonal type. Then there exist a polygon 
$\Delta\subset\Ddb$ and 
a constant $C\geq 1$ such that
$$
\norm{\phi(T_1,\ldots,T_d)}\leq C\norm{\phi}_{\infty,\Delta^d},
\qquad \phi\in\P_d.
$$
\end{theorem}

\begin{proof}
The scheme of proof will be similar to the one of
\cite[Theorem 6.2]{Ar} so we will be deliberately quick on some arguments.
We need ingredients (different from those in
\cite{Ar}) that we now introduce.

Let $E=\{\xi_1,\ldots,\xi_N\}\subset \Tdb$ such that
$T_k$ is a Ritt$_E$ operator for every $1\leq k\leq d$. 
We will use the $H^\infty$-functional calculus of
Ritt$_E$ operators, for which we refer to \cite[Section 3]{BLM}.
Let $H^\infty_{0}(\Ddb)$ be the space of all bounded analytic
functions $\phi\colon \Ddb\to\Cdb$  such that there exist
$s,c>0$ for which
$$
\vert\phi(z)\vert \leq c\prod_{j=1}^N\vert z-\xi_j\vert^s,\qquad z\in\Ddb.
$$
For any $1\leq k\leq d$ and any $\phi\in H^\infty_{0}(\Ddb)$,
$\phi(T_k)$ is a well-defined element of $B(H)$. Since 
$T_k$ is polynomially bounded, the argument in the proof
of  \cite[Proposition 2.5]{LM2} shows as well that 
\begin{equation}\label{5PB}
\norm{\phi(T_k)}\leq K\norm{\phi}_{\infty,{\mathbb D}},
\qquad \phi\in H^\infty_{0}(\Ddb),
\end{equation}
for some constant $K\geq 1$ not depending on $\phi$.

Let $F_0\in\P_d$ be defined by
$$
F_0(z_1,\ldots,z_d) = \prod_{k=1}^d\prod_{j=1}^N(z_k-\xi_j),
$$
and let $\P_{d,0}\subset\P_d$ be the ideal generated by $F_0$. Arguing as 
in the proof of \cite[Proposition 3.4]{BLM}, we see that to prove  Theorem
\ref{5Poly-FC}, it suffices to find a polygon 
$\Delta\subset\Ddb$ and 
a constant $C'\geq 1$ such that
\begin{equation}\label{5Goal}
\norm{\phi(T_1,\ldots,T_d)}\leq C'\norm{\phi}_{\infty,\Delta^d},
\qquad \phi\in\P_{d,0}.
\end{equation}

Recall that for any $r\in(0,1)$, we let $E_r$ be the interior of the convex hull of $E$
and of the disc $\{\lambda\in\Cdb\, :\,\vert\lambda\vert <r\}$.
According to  \cite[Proposition 3.4]{BLM}, 
there exist three sequences
$(\theta_i)_{i\geq 1}$, $(\phi_i)_{i\geq 1}$ and $(\psi_i)_{i\geq 1}$ of 
$H_{0}^\infty(\Ddb)$ such that:
\begin{align*}
&({\rm a})\ \sup\limits_{z \in {\mathbb D}} \sum\limits_{i = 1}^\infty\vert \phi_i(z)\vert <\infty 
\qquad\hbox{and}\qquad
\sup\limits_{z \in  {\mathbb D}} \sum\limits_{i = 1}^\infty \vert\psi_i(z)\vert <\infty;\qquad\qquad\\
&({\rm b}) \ \sup\limits_{i \geq 1}\, \sup\limits_{z \in  {\mathbb D}} 
\vert \theta_i(z)\vert <\infty;\qquad\qquad\\
&({\rm c})\ \forall\, r\in (0,1), 
\quad \sup\limits_{i\geq 1} \int_{\partial E_{r}} \frac {\vert \theta_i(z)\vert}
{\prod\limits_{j=1}^N \vert\xi_j -z\vert} \,\vert dz\vert< \infty;
\qquad\qquad\\
&({\rm d})\ \forall\, z \in \Ddb,\quad 1 = \sum\limits_{i = 1}^\infty \theta_i(z)\phi_i(z)\psi_i(z).\qquad\qquad
\end{align*}

We claim that there exists a constant $C>0$ such that
\begin{equation}\label{5SF}
\Bigl(\sum_{i=1}^\infty\norm{\phi_i(T_k)x}^2\Bigr)^\frac12\leq C\norm{x},
\end{equation}
for all $x\in H$ and $1\leq k\leq d$.
Indeed, for any $n\geq 1$, we have
$$
\Bigl(\sum_{i=1}^n\norm{\phi_i(T_k)x}^2\Bigr)^\frac12 \leq\sup_{\varepsilon_i=\pm 1}\Bignorm{
\sum_{i=1}^n \varepsilon_i \phi_i(T_k)x}.
$$
Furthermore for any $\varepsilon_i=\pm 1$, we have
$$
\Bignorm{\sum_{i=1}^n \varepsilon_i \phi_i(T_k)x}  \leq\Bignorm{\Bigl(\sum_{i=1}^n \varepsilon_i \phi_i\Bigr)(T_k)}\norm{x}
 \leq K\Bignorm{\sum_{i=1}^n \varepsilon_i \phi_i}_{\infty,{\mathbb D}}\norm{x}
\leq C_0 K\norm{x},
$$
where $K$ comes from (\ref{5PB}) and 
$C_0$ is the first supremum in (a).

Let $\phi\in \P_{d,0}$.
According to \cite[Lemma 2.8]{BLM}, there exists $r\in(0,1)$ such that
$\sigma(T_k)\subset E_r\cup E$ and 
\begin{equation}\label{5Unif}
\sup\Bigl\{\prod_{j=1}^N\vert \xi_j -z\vert \norm{R(z,T_k)}\, :\, z\in\partial E_r\setminus E\Bigr\}\,<\infty,
\end{equation}
for all $1\leq k\leq d$. Then for any $i_1,\ldots, i_d\geq 1$, we have
\begin{align*}
\phi(T_1&,\ldots,T_d)\prod_{k=1}^d \theta_{i_k}(T_k)\\
&=\Bigl(\frac{1}{2\pi i}\Bigr)^d \int_{(\partial E_r)^d} \phi(z_1,\ldots,z_d)\prod_{k=1}^d 
\bigl(\theta_{i_k}(z_k) R(z_k,T_k)\bigr)\,dz_1\cdots dz_d.
\end{align*}
Thanks to the uniform estimate (\ref{5Unif}) and to (c), this integral representation implies that
$$
\Bignorm{\phi(T_1,\ldots,T_d)\prod_{k=1}^d \theta_{i_k}(T_k)}\leq M\norm{\phi}_{\infty,(\partial E_r)^d},
$$
for some constant $M>0$ not depending on $\phi$.

Writing (d) for independent variables $z_1,\ldots, z_d$ and multiplying the resulting  identities, we obtain that
$$
1 = \sum_{i_1,\ldots,i_d=1}^{\infty}\, \prod_{k=1}^d \theta_{i_k}(z_k)
\prod_{k=1}^d \phi_{i_k}(z_k)\prod_{k=1}^d \psi_{i_k}(z_k),
\qquad (z_1,\ldots,z_d)\in\Ddb^d.
$$
Then by an entirely classical Fubini argument based on (\ref{5Unif}), we deduce  that
\begin{equation}\label{5Lim}
\phi(T_1,\ldots,T_d) =  \sum_{i_1,\ldots,i_d=1}^{\infty}\, \phi(T_1,\ldots,T_d)\prod_{k=1}^d \theta_{i_k}(T_k)
\prod_{k=1}^d \phi_{i_k}(T_k)\prod_{k=1}^d \psi_{i_k}(T_k),
\end{equation}
the family on the right hand side being absolutely summable in $B(H)$.
In the sequel, we set
\begin{equation}\label{5Lambda}
\Lambda_n =  \sum_{i_1,\ldots,i_d=1}^{n}\, \phi(T_1,\ldots,T_d)\prod_{k=1}^d \theta_{i_k}(T_k)
\prod_{k=1}^d \phi_{i_k}(T_k)\prod_{k=1}^d \psi_{i_k}(T_k),
\end{equation}
for all $n\geq 1$.

Let $x,y\in H$ and let $n\geq 1$. By the Cauchy-Schwarz inequality, we have
\begin{align*}
\bigl\vert \langle \Lambda_n x,y\rangle \bigr\vert & \leq \sup_{1\leq i_1,\ldots, i_d\leq n}
\Bignorm{\phi(T_1,\ldots,T_d)\prod_{k=1}^d \theta_{i_k}(T_k)}\\
& \times \biggl(\sum_{i_1,\ldots,i_d=1}^{n}\Bignorm{\prod_{k=1}^d \phi_{i_k}(T_k)x}^2\biggr)^\frac12
\biggl(\sum_{i_1,\ldots,i_d=1}^{n}\Bignorm{\prod_{k=1}^d \psi_{i_k}(T_k)^* y}^2\biggr)^\frac12,
\end{align*}
which is less than or equal to 
$$
M\norm{\phi}_{\infty,(\partial E_r)^d}
\biggl(\sum_{i_1,\ldots,i_d=1}^{n}\Bignorm{\prod_{k=1}^d \phi_{i_k}(T_k)x}^2\biggr)^\frac12
\biggl(\sum_{i_1,\ldots,i_d=1}^{n}\Bignorm{\prod_{k=1}^d \psi_{i_k}(T_k)^* y}^2\biggr)^\frac12.
$$
Applying (\ref{5SF}) $d$ times, we obtain that 
$$
\biggl(\sum_{i_1,\ldots,i_d=1}^{n}\Bignorm{\prod_{k=1}^d \phi_{i_k}(T_k)x}^2\biggr)^\frac12
\leq C^d\norm{x}.
$$
Using the second supremum in (a), we obtain as well an estimate
$$
\biggl(\sum_{i_1,\ldots,i_d=1}^{n}\Bignorm{\prod_{k=1}^d \psi_{i_k}(T_k)^* y}^2\biggr)^\frac12
\leq C^d\norm{y}.
$$
Hence we finally have
$$
\bigl\vert \langle\Lambda_n x,y\rangle \bigr\vert\leq 
M C^{2d}\norm{\phi}_{\infty,(\partial E_r)^d}\norm{x}\norm{y}.
$$
Taking the supremum over all $x,y$ with $\norm{y}=\norm{x}=1$ and applying (\ref{5Lim}) and (\ref{5Lambda}), we deduce the estimate
$$
\bignorm{\phi(T_1,\ldots,T_d)}\leq 
M C^{2d}\norm{\phi}_{\infty,(\partial E_r)^d}.
$$
Now choose a polygon $\Delta$ such that
$E_r\subset\Delta\subset\Ddb$. Then we obtain (\ref{5Goal}) with $C'=M C^{2d}$.
\end{proof}

\bigskip
\noindent
{\bf Acknowledgements.} 
The authors warmly thank the referee for his/her careful reading of the first version of this manuscript.
The first author was supported by the ANR project Noncommutative
analysis on groups and quantum groups (No./ANR-19-CE40-0002).

\end{document}